\documentclass[11pt,letterpaper]{article}

\usepackage{tikz}
\usepackage[margin=1in]{geometry}

\renewcommand{\phi}{\varphi}

\newcommand{\R}{\mathbb{R}}

\def\ds1{\mathds{1}}
\renewcommand{\epsilon}{\varepsilon}

\newcommand{\argmin}{\mathop{\mathrm{arg\,min}}}

\renewcommand{\tilde}{\widetilde}

\newlength{\minipagewidth}
\newcommand{\beq}{\begin{equation}}
\newcommand{\eeq}{\end{equation}}

\newcommand{\beqa}{\begin{eqnarray}}
\newcommand{\eeqa}{\end{eqnarray}}

\newcommand{\beqan}{\begin{eqnarray*}}
\newcommand{\eeqan}{\end{eqnarray*}}

\def\ba#1\ea{\begin{align*}#1\end{align*}}
\def\banum#1\eanum{\begin{align}#1\end{align}}

\usepackage[lined,boxed,ruled,norelsize,algo2e,linesnumbered]{algorithm2e}
\usepackage[hidelinks]{hyperref}
\usepackage{url}
\usepackage[numbers]{natbib}
\usepackage{graphicx,subfigure,amsmath,amssymb,amsfonts,bm,epsfig,epsf,url,dsfont,color}
\usepackage{algorithm}
\usepackage{algorithmic}
\usepackage{xcolor}

\usepackage{amsthm}
\newtheorem{theorem}{Theorem}[section]

\newtheorem{lemma}[theorem]{Lemma}

\newtheorem{remark}[theorem]{Remark}

\title{Near-optimal method for highly smooth convex optimization}

\def\And{\and}
\author{S\'ebastien Bubeck\\
       Microsoft Research\\
			 \And
			 Qijia Jiang \\
       Stanford University \\
			 \And
Yin Tat Lee \thanks{Research was supported in part by NSF Awards CCF-1740551, CCF-1749609, and DMS-1839116.}\\
       University of Washington \\ \& Microsoft Research\\
			 \And
       Yuanzhi Li  \\
      Stanford University\\
			 \And
      Aaron Sidford \thanks{Research was supported in part by NSF CAREER Award CCF-1844855.}\\
      Stanford University \\
}

\begin{document}

\maketitle

\begin{abstract}
We propose a near-optimal method for highly smooth convex optimization. More precisely, in the oracle model where one obtains the $p^{th}$ order Taylor expansion of a function at the query point, we propose a method with rate of convergence $\tilde{O}(1/k^{\frac{ 3p +1}{2}})$ after $k$ queries to the oracle for any convex function whose $p^{th}$ order derivative is Lipschitz. 
\end{abstract}

\section{Introduction}
In this paper we generalize the important phenomenon of {\em acceleration} in smooth convex optimization \cite{NY83, Nem82, Nes83} to higher orders of smoothness. We consider a $p^{th}$-order Taylor expansion oracle, that is given a query point $x \in \R^d$ it returns a $p^{th}$ order Taylor expansion of the objective function $f$ at the point $x$:
$$
f_{p}(y,x)=f(x)+\sum_{i=1}^{p}\frac{1}{i!}\nabla^{i}f(x)[y-x]^{i}.
$$
We propose a new optimization method based on such oracle, see Algorithm 1, which we term {\em accelerated Taylor descent (ATD)}. We prove that it attains a nearly optimal rate of convergence under higher order smoothness (the matching lower bounds were recently proven in \cite{AH18, ASS18}), namely after $\tilde{O}(k)$ calls to the oracle it achieves error $O(1/k^{\frac{ 3p +1}{2}})$. This improves upon the $O(1/k^{p+1})$ derived in \cite{Nes18} (both rates match for $p=1$, i.e., the classical acceleration setting), and it matches the rate given in \cite{MS13} for $p=2$.
\begin{theorem} \label{thm:main}
 Let $f$ denote a convex function whose $p^{th}$ derivative is $L_p$-Lipschitz and let $x^*$ denote a minimizer of $f$. Then ATD satisfies, with $c_p = 2^{p-1} (p+1)^{\frac{3p+1}{2}} / (p-1)!$,
 \begin{equation} \label{eq:firststatement}
 f(y_k) - f(x^*) \leq \frac{c_p \cdot L_p \cdot \|x^*\|^{p+1}}{k^{\frac{ 3p +1}{2}}} \,.
 \end{equation}
Furthermore each iteration of ATD can be implemented in $\tilde{O}(1)$ calls to a $p^{th}$-order Taylor expansion oracle. More precisely, given a precision $\epsilon >0$, at each iteration $k$, using at most 
\[
30p\log_{2}p+\log_{2}\left\lceil \frac{L_{p}\|x^{*}\|^{p+1}}{\epsilon}\right\rceil
\] 
calls to the $p^{th}$-order Taylor expansion oracle we find either a point $y$ such that $f(y) - f(x^*) \leq \epsilon$, or we find $y_{k}$.
\end{theorem}
Our method is largely inspired by \cite{MS13},
which
focuses on $p=2$
, and we recall their framework in Section \ref{sec:MS}. We then specialize this framework to higher order smoothness in Section \ref{sec:ATD}, where we derive and analyze ATD. A subtle point of ATD is that an iteration requires more than one call to the oracle due to the ``line-search'' [line 4, Algorithm 1]. We prove that $\tilde{O}(1)$ calls suffice to implement an iteration in Section \ref{sec:implementation}.

We note that the independent work \cite{Gasnikov18}, currently only available in Russian, derive a similar result to \eqref{eq:firststatement}. From our understanding of their work it seems however that they do not work out the precise complexity of the binary search step (second part of the statement in Theorem \ref{thm:main}, see also Section \ref{sec:implementation}). Finally we note that yet another independent work \cite{JWZ18} was posted on the arxiv a couple of days prior to us, with a similar result to Theorem \ref{thm:main}. Interestingly it seems that their argument to control the complexity of the binary search is different (at least on the surface) from ours.

\begin{algorithm} [h!]
\caption{Accelerated Taylor Descent}\label{alg:highorder}
	\begin{algorithmic}[1]
		\STATE \textbf{Input:} convex function $f : \R^d \rightarrow \R$ such that $\nabla^p f$ is $L_p$-Lipschitz.
		\STATE Set $A_0 = 0, x_0 = y_0 = 0$
		\FOR{ $k = 0$ \TO $k = K- 1$ }
		\STATE Compute a pair $\lambda_{k+1} > 0$ and $y_{k+1}\in \R^d$ such that
		\[
\frac{1}{2} \leq \lambda_{k+1} \frac{L_p \cdot \|y_{k+1} - \tilde{x}_k\|^{p-1}}{(p-1)!}  \leq \frac{p}{p+1} \,,
\]
where
\[
y_{k+1} = \argmin_y \left\{
f_p(y;\tilde{x}_k)+\frac{L_p}{p!}\|y-\tilde{x}_k\|^{p+1}\right\} \,,
\]
and
		\[
a_{k+1} = \frac{\lambda_{k+1}+\sqrt{\lambda_{k+1}^2+4\lambda_{k+1}A_k}}{2} 
\text{ , } 
A_{k+1} = A_k+a_{k+1}
\text{ , and } 
\tilde{x}_k = \frac{A_k}{A_{k + 1}}y_k + \frac{a_{k+1}}{A_{k+1}} x_k \,.
		\]
		\STATE Update $x_{k+1} := x_k-a_{k+1} \nabla f(y_{k+1})$
		\ENDFOR
		\RETURN $y_{K}$ 
	\end{algorithmic}	
\end{algorithm}

\begin{remark}
The definition of $a_{k+1}$ was chosen such that $\lambda_{k + 1} A_{k + 1} = a_{k+1}^2$. To see this, note that $a_{k + 1}$ is a solution to $a_{k+1}^2 - \lambda_{k + 1} a_{k+1} - \lambda_{k + 1} A_k = 0$, which is equivalent as $A_{k + 1} = A_k + a_k$.
\end{remark}

\section{Monteiro-Svaiter acceleration framework} \label{sec:MS}
Recall that Nesterov's accelerated gradient descent \cite{Nes83, Nes04} produces a sequence of the form:
\begin{equation} \label{eq:agd}
y_{k+1} = \tilde{x}_k - \lambda_{k+1} \nabla f(\tilde{x}_k) \,,
\end{equation}
for some step size $\lambda_{k+1}$ and ``momentum'' point $\tilde{x}_k$. In this section we consider a variant proposed by Monteiro and Svaiter which replaces the gradient step by a form of ``implicit gradient step'', namely:
\[
y_{k+1} \simeq \tilde{x}_k - \lambda_{k+1} \nabla f(y_{k+1}) \,.
\]
The rest of the section is merely a rewriting of \cite{MS13}, with the objective to motivate and prove the following result:
\begin{theorem} \label{thm:MS}
Let $(y_k)_{k \geq 1}$ be a sequence of points in $\R^d$ and $(\lambda_k)_{k \geq 1}$ a sequence in $\R_+$. Define $(a_k)_{k \geq 1}$ such that $\lambda_k A_k = a_k^2$ where $A_k = \sum_{i=1}^k a_i$. Define also for any $k\geq 0$, $x_k = - \sum_{i=1}^k a_i \nabla f(y_i)$ (in particular $x_0 = 0$) and $\tilde{x}_k := \frac{a_{k+1}}{A_{k+1}} x_{k} + \frac{A_k}{A_{k+1}} y_k$. Finally assume that
\begin{equation} \label{eq:igd}
\|y_{k+1} - (\tilde{x}_k - \lambda_{k+1} \nabla f(y_{k+1}))\| \leq \|y_{k+1} - \tilde{x}_k\| \,.
\end{equation}
Then one has for any $x \in \R^d$,
\begin{equation} \label{eq:rate}
f(y_k) - f(x) \leq \frac{2 \|x\|^2}{\left(\sum_{i=1}^k \sqrt{\lambda_i} \right)^2}  \,.
\end{equation}
Furthermore if one has the following refined guarantee, for some $\sigma \in [0,1]$,
\begin{equation} \label{eq:igdrefined}
\|y_{k+1} - (\tilde{x}_k - \lambda_{k+1} \nabla f(y_{k+1}))\| \leq \sigma \cdot \|y_{k+1} - \tilde{x}_k\| \,,
\end{equation}
then one also has
\begin{equation} \label{eq:Alambdatradeoff}
\sum_{i=1}^k \frac{A_i}{\lambda_i} \|y_i - \tilde{x}_{i-1}\|^2 \leq \frac{\|x^*\|^2}{1-\sigma^2} \,.
\end{equation}
\end{theorem}

To illustrate the power of Theorem \ref{thm:MS}, observe that for a $L_1$-smooth function (first-order smoothness) one has that Nesterov's accelerated gradient descent \eqref{eq:agd} directly satisfies \eqref{eq:igd} provided that $\lambda_{k+1}= \frac{1}{L_1}$ (i.e., the classical step-size for smooth convex optimization). Using \eqref{eq:rate} this immediately shows that \eqref{eq:agd} has a rate of convergence of $O(1/k^2)$

The key to higher-order acceleration will be to show that in fact one can take $\lambda_k$ to be an increasing function of $A_k$, thanks 
 to a careful use of \eqref{eq:Alambdatradeoff}. This will be done in Section~\ref{sec:ATD}.

We now embark on the road leading to Theorem \ref{thm:MS}.

\subsection{Estimate sequence style analysis}
Similarly to the original construction by Nemirovski \cite{NY83, Nem82} (and taking inspiration from the conjugate gradient method) the starting point is to consider a linear combination of past gradients: $x_k := - \sum_{i=1}^k a_i \nabla f(y_i)$, where both the coefficients $(a_i)$ and the query points $(y_i)$ are yet to be defined. In the spirit of Nesterov's estimate sequence analysis, a key observation for such linear combination of gradients is that it minimizes an approximate lower bound on $f$:
\begin{lemma} \label{lem:basic1}
Let $\psi_0(x) = \frac{1}{2} \|x\|^2$ and define by induction $\psi_{k}(x) = \psi_{k-1}(x) + a_{k} f_1(x, y_{k})$. Then $x_k = - \sum_{i=1}^k a_i \nabla f(y_i)$ is the minimizer of $\psi_k$, and 
$\psi_k(x) \leq A_k f(x) + \frac{1}{2} \|x\|^2$ where $A_k = \sum_{i=1}^k a_i$. 
\end{lemma}

The next idea is to produce a "control sequence'' $(z_k)_{k\geq1}$ demonstrating that $\psi_k$ is not too far below $A_k f$, which in turn would directly yield a convergence rate for $z_k$ of order $1/A_k$:

\begin{lemma} \label{lem:basic2}
Let $(z_k)$ be a sequence such that 
\begin{equation} \label{eq:tosatisfy}
\psi_k(x_k) - A_k f(z_k) \geq 0 \,.
\end{equation}
Then one has for any $x$,
\begin{equation} \label{eq:tosatisfy2}
f(z_k) \leq f(x) + \frac{\|x\|^2}{2 A_k} \,.
\end{equation}
\end{lemma}

\begin{proof}
One has (recall Lemma \ref{lem:basic1}):
\[
A_k f(z_k) \leq \psi_k(x_k) \leq \psi_k(x) \leq A_k f(x) + \frac{1}{2}\|x\|^2 \,.
\]
\end{proof}

\subsection{A proof by induction}
Our goal is now to come up with sequences $(a_k, y_k, z_k)$ satisfying \eqref{eq:tosatisfy}. The following lemma, resulting from 
elementary calculations, reveals a simple condition to obtain \eqref{eq:tosatisfy} from an induction argument:
\begin{lemma} \label{lem:basic3}
One has for any $x$,
\begin{align*}
& \psi_{k+1}(x) - A_{k+1} f(y_{k+1}) - (\psi_k(x_k) - A_k f(z_k)) \\
& \geq A_{k+1} \nabla f(y_{k+1}) \cdot \left(\frac{a_{k+1}}{A_{k+1}} x + \frac{A_k}{A_{k+1}} z_k - y_{k+1} \right ) + \frac{1}{2} \|x -x_k\|^2 \,.
\end{align*}
\end{lemma}

\begin{proof}
First we note that (the first equality follows from the fact that the Hessian of $\psi_k$ remains the identity for any $k$):
\[
\psi_k(x) = \psi_k(x_k) + \frac{1}{2} \|x- x_k\|^2, \text{ and } \psi_{k+1}(x) = \psi_k(x_k) + \frac{1}{2} \|x-x_k\|^2 + a_{k+1} f_1(x, y_{k+1}) \,,
\]
so that 
\begin{equation} \label{eq:ind1}
\psi_{k+1}(x) - \psi_k(x_k) = a_{k+1} f_1(x, y_{k+1}) + \frac{1}{2} \|x-x_k\|^2 \,.
\end{equation}
Now we want to make appear the term $A_{k+1} f(z_{k+1}) - A_k f(z_k)$ as a lower bound on the right hand side of \eqref{eq:ind1} when evaluated at $x=x_{k+1}$. 
Using the inequality 
 $f_1(z_k, y_{k+1}) \leq f(z_k)$
 we have:
\begin{eqnarray*}
a_{k+1} f_1(x, y_{k+1}) & = & A_{k+1} f_1(x, y_{k+1}) - A_k f_1(x, y_{k+1}) \\
& = & A_{k+1} f_1(x, y_{k+1}) - A_k \nabla f(y_{k+1}) \cdot (x - z_k) - A_k f_1(z_k, y_{k+1}) \\
 & = & A_{k+1} f_1\left(x - \frac{A_k}{A_{k+1}} (x - z_k), y_{k+1} \right ) - A_k f_1(z_k, y_{k+1}) \\
 & \geq & A_{k+1} f(y_{k+1}) - A_k f(z_k) + A_{k+1} \nabla f(y_{k+1}) \cdot \left(\frac{a_{k+1}}{A_{k+1}} x + \frac{A_k}{A_{k+1}} z_k - y_{k+1} \right ) \,,
\end{eqnarray*}
which concludes the proof.
\end{proof}

From Lemma \ref{lem:basic3} we see that it is natural to take for the control sequence $z_k := y_k$, so that:
\begin{align}
& \psi_{k+1}(x) - A_{k+1} f(y_{k+1}) - (\psi_k(x_k) - A_k f_k(y_k)) \\
& \geq  A_{k+1} \nabla f(y_{k+1}) \cdot \left(\frac{a_{k+1}}{A_{k+1}} x + \frac{A_k}{A_{k+1}} y_k - y_{k+1} \right )  + \frac{1}{2} \|x - x_k\|^2 \,. \label{eq:rhs}
\end{align} 
We would like to pick the query point $y_{k+1}$ so that \eqref{eq:rhs} is nonnegative when evaluated at $x= x_{k+1}$ (to satisfy \eqref{eq:tosatisfy}). One difficulty is that $x_{k+1}$ itself depends on $y_{k+1}$, so in fact we will pick $y_{k+1}$ so that the right side is nonnegative {\em for all $x$}. We write this as follows:

\begin{lemma} \label{lem:basic4}
Denoting $\lambda_{k+1} := \frac{a_{k+1}^2}{A_{k+1}}$ and $\tilde{x}_k := \frac{a_{k+1}}{A_{k+1}} x_{k} + \frac{A_k}{A_{k+1}} y_k$ one has:
\begin{align*}
& \psi_{k+1}(x_{k+1}) - A_{k+1} f(y_{k+1}) - (\psi_k(x_k) - A_k f(y_k)) \\
& \geq \frac{A_{k+1}}{2 \lambda_{k+1}} \bigg( \|y_{k+1} - \tilde{x}_k\|^2 - \|y_{k+1} - (\tilde{x}_k - \lambda_{k+1} \nabla f(y_{k+1})) \|^2 \bigg) \,.
\end{align*}
In particular, we have in light of \eqref{eq:igdrefined}
$$\psi_{k}(x_{k})-A_{k}f(y_{k})\geq\frac{1-\sigma^{2}}{2}\sum_{i=1}^{k}\frac{A_{i}}{\lambda_{i}}\|y_{i}-\tilde{x}_{i-1}\|^{2}.$$
\end{lemma}

\begin{proof}
We apply Lemma \ref{lem:basic3} with $z_k = y_k$ and $x=x_{k+1}$, and 
note that (with $\tilde{x} := \frac{a_{k+1}}{A_{k+1}} x + \frac{A_k}{A_{k+1}} y_k$): 
\begin{align*}
& \nabla f(y_{k+1}) \cdot \left(\frac{a_{k+1}}{A_{k+1}} x + \frac{A_k}{A_{k+1}} y_k - y_{k+1} \right )  + \frac{1}{2 A_{k+1}} \|x - x_k\|^2 \\
& = \nabla f(y_{k+1}) \cdot (\tilde{x} - y_{k+1}) + \frac{1}{2 A_{k+1}} \left\|\frac{A_{k+1}}{a_{k+1}} \left(\tilde{x} - \frac{A_k}{A_{k+1}} y_k \right) - x_k \right\|^2 \\
& = \nabla f(y_{k+1}) \cdot (\tilde{x} - y_{k+1}) + \frac{A_{k+1}}{2 a_{k+1}^2} \left\|\tilde{x} - \left(\frac{a_{k+1}}{A_k} x_k + \frac{A_k}{A_{k+1}} y_k \right) \right\|^2 \,.
\end{align*}
This yields:
\begin{align*}
& \psi_{k+1}(x_{k+1}) - A_{k+1} f(y_{k+1}) - (\psi_k(x_k) - A_k f(y_k)) \\
& \geq A_{k+1} \cdot \min_{x \in \R^d} \left\{ \nabla f(y_{k+1}) \cdot (x - y_{k+1}) + \frac{1}{2 \lambda_{k+1}} \|x - \tilde{x}_k\|^2 \right\} \,.
\end{align*}
It only remains to compute the value of this minimum, which is an easy exercise.
\end{proof}

\subsection{Proof of Theorem \ref{thm:MS}}
For the first conclusion in Theorem \ref{thm:MS}, it suffices to combine Lemma \ref{lem:basic4} with Lemma \ref{lem:basic2}, and to use the following observation:
\begin{lemma} \label{lem:Aestimate}
Let $(\lambda_k)$ be a sequence of nonnegative numbers. Define $(a_k)$ to be another sequence of nonnegative numbers such that $\lambda_k A_k = a_k^2$, where $A_k = \sum_{i =1}^k a_i$. In other words one has $a_k = \frac{\lambda_k + \sqrt{\lambda_k^2 + 4 \lambda_k A_{k-1}}}{2}$. Furthermore one also has:
\[
\sqrt{A_k} \geq \frac{1}{2} \sum_{i=1}^k \sqrt{\lambda_i} \,. 
\]
\end{lemma}
\begin{proof}
It suffices to observe that:
\[
a_k = \frac{\lambda_k + \sqrt{\lambda_k^2 + 4 \lambda_k A_{k-1}}}{2} \geq \frac{\lambda_k}{2} + \sqrt{\lambda_k A_{k-1}} \geq \left(\frac{\sqrt{\lambda_k}}{2} + \sqrt{A_{k-1}} \right)^2 - A_{k-1} \,.
\]
\end{proof}
The second conclusion in Theorem \ref{thm:MS} follows from Lemma \ref{lem:basic4} and Lemma \ref{lem:basic1}.

\section{Accelerated Taylor Descent} \label{sec:ATD}
Nesterov's accelerated gradient descent \eqref{eq:agd} (with $\lambda_{k} = 1/L_1$) can be rewritten as:
\[
y_{k+1} = \argmin_{y \in \R^d} f_1(y, \tilde{x}_k) + \frac{L_1}{2} \|y - \tilde{x}_k\|^2 \,.
\]
We naturally propose to use the following generalization for higher-order smoothness, which we term {\em accelerated Taylor descent (ATD)}:
\begin{equation} \label{eq:atddef}
y_{k+1} = \argmin_{y \in \R^d} f_p(y, \tilde{x}_k) + \frac{L_p}{p!} \|y - \tilde{x}_k\|^{p+1} \,.
\end{equation}
The term $\| \cdot \|^{p+1}$  is added to ensure that the function being optimized is strictly convex.
In 
Section~\ref{sec:atdigd} we first show that ATD satisfies \eqref{eq:igd} for a special value of $\lambda_{k+1}$ defined in terms of $y_{k+1}$. We point out that there is an intricate issue here, in the sense that $y_{k+1}$ depends on $\lambda_{k+1}$ (through the definition of $\tilde{x}_k$), and thus we will have to select the the pair $(y_{k+1}, \lambda_{k+1})$ simultaneously rather than sequentially. This is detailed in 
Section~\ref{sec:continuity}. Finally in 
Section~\ref{sec:proofmain} 
we use \eqref{eq:Alambdatradeoff} with the special values of $(\lambda_i)$ to derive the rate of convergence from Theorem \ref{thm:main}.

\subsection{ATD and implicit gradient descent with large step size} \label{sec:atdigd}
The following lemma shows that minimizing the $p^{th}$ order Taylor expansion \eqref{eq:atddef} can be viewed as an implicit gradient step for some ``large'' step size:
\begin{lemma} \label{lem:controlstepsize}
Equation \eqref{eq:igdrefined} holds true with $\sigma = 1/2$ for \eqref{eq:atddef}, provided that one has: 
\begin{equation} \label{eq:key4}
\frac{1}{2} \leq \lambda_{k+1} \frac{L_p \cdot \|y_{k+1} - \tilde{x}_k\|^{p-1}}{(p-1)!}  \leq \frac{p}{p+1} \,.
\end{equation}
\end{lemma}

\begin{proof}
Observe that the optimality condition gives:
\begin{equation} \label{eq:KKT_TD}
\nabla_y f_p(y_{k+1}, \tilde{x}_k) + \frac{L_p \cdot (p+1)}{p!} (y_{k+1} - \tilde{x}_k) \|y_{k+1} - \tilde{x}_k\|^{p-1} = 0 \,. 
\end{equation}
In particular we get:
\[
y_{k+1} - (\tilde{x}_k - \lambda_{k+1} \nabla f(y_{k+1}))  = \lambda_{k+1} \nabla f(y_{k+1}) - \frac{p!}{L_p \cdot (p+1) \cdot \|y_{k+1} - \tilde{x}_k\|^{p-1}} \nabla_y f_p(y_{k+1}, \tilde{x}_k) \,.
\]
By doing a Taylor expansion of the gradient function one obtains:
\[
\|\nabla f(y) - \nabla_y f_p(y, x)\| \leq \frac{L_p}{p!} \|y - x\|^p \,,
\]
so that we find:
\begin{align*}
& \|y_{k+1} - (\tilde{x}_k - \lambda_{k+1} \nabla f(y_{k+1})) \| \\
& \leq \lambda_{k+1} \frac{L_p}{p!} \|y_{k+1} - \tilde{x}_k\|^p + \left|\lambda_{k+1} - \frac{p!}{L_p \cdot (p+1) \cdot \|y_{k+1} - \tilde{x}_k\|^{p-1}} \right| \cdot \|\nabla_y f_p(y_{k+1}, \tilde{x}_k)\| \\
& \leq \|y_{k+1} - \tilde{x}_k\| \left(\lambda_{k+1} \frac{L_p}{p!} \|y_{k+1} - \tilde{x}_k\|^{p-1} + \left|\lambda_{k+1}\frac{L_p \cdot (p+1) \cdot \|y_{k+1} - \tilde{x}_k\|^{p-1}}{p!} - 1\right|  \right) \\
&=\|y_{k+1}-\tilde{x}_{k}\|\left(\frac{\eta}{p}+\left|\eta\cdot\frac{p+1}{p}-1\right|\right)
\end{align*}
where we used \eqref{eq:KKT_TD} in the second last equation and we let $\eta := \lambda_{k+1} \frac{L_p \cdot \|y_{k+1} - \tilde{x}_k\|^{p-1}}{(p-1)!}$ in the last equation.  The result follows from the assumption $1/2 \leq \eta \leq p/(p+1)$ in \eqref{eq:key4}. 
\end{proof}

\subsection{A continuity argument} \label{sec:continuity}
We now claim that there exists a pair $(y_{k+1}, \lambda_{k+1})$ that satisfies simultaneously \eqref{eq:atddef} and \eqref{eq:key4}. This is a direct consequence of the following lemma.
\begin{lemma}
Let $A \geq 0$, $x, y \in \R^d$ such that $f(x) \neq f(x^*)$. Define the following functions:
\begin{align*}
& a(\lambda) = \frac{\lambda + \sqrt{\lambda^2 + 4 \lambda A}}{2} \,, \, x(\lambda) = \frac{a(\lambda)}{A + a(\lambda)} x + \frac{A}{A+a(\lambda)} y \,, \\
& y(z) = \argmin_{w \in \R^d} \left\{ f_p(w, z) + \frac{L_p}{p!} \|w - z\|^{p+1} \right\} \,, \, g(\lambda) = \lambda \|y(x(\lambda)) - x(\lambda)\|^{p-1} \,.
\end{align*}
Then we have $g(\R_+) = \R_+$.
\end{lemma}

\begin{proof}
First we claim that $g(\lambda)$ is a continuous function of $\lambda$. The only non-trivial part of this statement is that $y(z)$ is a continuous function of $z$. The latter statement follows easily from the strict convexity of the function being optimized, see also Section~\ref{sec:implementation} for more details.

Next we claim that $g(0) = 0$, and furthermore since $f(x) \neq f(x^*)$ we also have $y(x) \neq x$ which in turns gives $g(+\infty) = +\infty$. This concludes the proof.
\end{proof}

\subsection{Proof of \eqref{eq:firststatement} in Theorem \ref{thm:main}} \label{sec:proofmain}
Recall from Lemma \ref{lem:basic2} that the rate of convergence of ATD is $\|x^*\|^2 / (2 A_k)$. We now finally give an estimate of $A_k$:
\begin{lemma} \label{lem:Abound}
One has, with $c_p = 2^{p-1} (p+1)^{\frac{3p+1}{2}} / (p-1)!$,
\[
A_k \geq \frac{1}{c_p \cdot L_p \cdot \|x^*\|^{p-1}} k^{\frac{3p+1}{2}} \,.
\]
\end{lemma}
\begin{proof}
Using Lemma \ref{lem:controlstepsize} (and in particular \eqref{eq:key4}) in \eqref{eq:Alambdatradeoff} we obtain, with $C_p = 8 \cdot \left( \frac{L_p}{(p-1)!} \right)^{\frac{2}{p-1}}$,
\begin{equation} \label{eq:rec1}
\sum_{i=1}^k \frac{A_i}{\lambda_{i}^{\frac{p+1}{p-1}}} \leq C_p \|x^*\|^2 \,.
\end{equation}
Now by reverse H\"{o}lder inequality, i.e. 
 $\|fg\|_1 \geq \|f\|_{\frac{1}{q}} \|g\|_{\frac{-1}{q-1}}$ for $q \geq 1$, and invoking this inequality with  $q = 1 + \frac{p - 1}{2(p + 1)} = \frac{3p + 1}{2(p + 1)}$ so that $\frac{-1}{1 - q}=-\frac{2(p + 1)}{p - 1}$, we have
\begin{equation}
\label{eq:rec2}
\sum_{j =1}^k \sqrt{\lambda_j}
= 
\sum_{j =1}^k 
\left(
A_j
\right)^{\frac{p - 1}{2(p + 1)}}
\left(
\frac{A_j}{ \lambda_j^{\frac{p + 1}{p - 1} } }
\right)^{- \frac{p - 1}{2(p + 1)}}
\geq 
\left(
\sum_{j =1}^k A_j^{\frac{p - 1}{3p + 1}}
\right)^{\frac{3p + 1}{2(p + 1)}}
\left(
\sum_{j =1}^k \frac{A_j}{ \lambda_j^{\frac{p + 1}{p - 1} } }
\right)^{-\frac{p - 1}{2(p + 1)}} ~.
\end{equation}
Combining \eqref{eq:rec1} and \eqref{eq:rec2} and using by Lemma \ref{lem:Aestimate} we have for all $k \geq 1$ that
\[
A_k \geq \frac{1}{4} \left(\sum_{j \in [k]} \sqrt{\lambda_j}\right)^2
\geq \frac{1}{4 (C_p \|x^*\|^2)^{\frac{p-1}{p+1}}} 
\left(
\sum_{j =1}^k A_j^{\frac{p-1}{3p+1}}
\right)^{\frac{3p+1}{p+1}}
\] 
Next we apply Lemma \ref{lem:recursion} (see below) with $\alpha= \frac{p+1}{p-1}$, $B_k = A_k^{\frac{p-1}{3p+1}}$ and $c=\frac{1}{4^{\frac{p+1}{3p+1}} (C_p \|x^*\|^2)^{\frac{p-1}{3p+1}}}$:
\[
B_k \geq \left( \frac{2}{p+1} \cdot c \cdot k \right)^{\frac{p-1}{2}} \,,
\]
or in other words, $A_k \geq \left( \frac{2}{p+1} \cdot c \cdot k \right)^{\frac{3p+1}{2}}$, which concludes the proof.
\end{proof}

\begin{lemma}\label{lem:recursion}
Given a non-decreasing positive sequence $B_{j}$ such that
$B_{k}^{\alpha}\geq c\cdot\sum_{j=1}^{k}B_{j}$.
Then, we have that
\[
B_{k}\geq\left(\frac{\alpha-1}{\alpha}c\cdot k\right)^{\frac{1}{\alpha-1}}
\]
\end{lemma}
\begin{proof}
We extend $B_{t}=B_{\left\lceil t\right\rceil }$. Note that
\[
B_{t}^{\alpha}=B_{\left\lceil t\right\rceil }^{\alpha}\geq c\cdot\sum_{j=1}^{\left\lceil t\right\rceil }B_{j}\geq c\cdot\int_{0}^{t}B_{s}ds.
\]
We can upper bound this integral inequality $B_{t}\geq U_{t}$ where $U_1=B_1$ and
\[
U_{t}^{\alpha}=c\cdot\int_{0}^{t}U_{s}ds.
\]
Taking derivatives on both sides, we have
$$
\alpha U_{t}^{\alpha-1}\frac{dU_{t}}{dt}=c\cdot U_{t}.
$$ and hence $\frac{dU_{t}^{\alpha-1}}{dt}=\frac{\alpha-1}{\alpha}c$.
Therefore, we have $B_t \geq U_{t}=(\frac{\alpha-1}{\alpha}c\cdot(t-1)+B_{1}^{\alpha-1})^{\frac{1}{\alpha-1}}$. Finally, the result follows from $B_1^{\alpha-1} \geq c$.
\end{proof}

\section{Complexity of the binary search step} \label{sec:implementation}
In this section, we show how to find $\lambda_{k+1}$ satisfying equation
\eqref{eq:key4}. For $k=0$, it is trivial since $\tilde{x}_0 = 0$.
From now on, we fix some $k>0$. To simplify the notation, we define
$\widetilde{x}_{\theta}=(1-\theta)x_{k}+\theta y_{k}$, $y_{\theta}=\argmin_{y}F(y-\widetilde{x}_{\theta},\widetilde{x}_{\theta})$
with
\[
F(z,x)=f_{p}(x+z,x)+\frac{L_{p}}{p!}\|z\|^{p+1},
\]
and $z_{\theta}=y_{\theta}-\widetilde{x}_{\theta}$. Note that the
$\lambda_{k+1}$ corresponding to $\theta$ is given by $\lambda_{k+1}=\frac{(1-\theta)^{2}}{\theta}A_{k}.$
Hence, our goal is to find $\theta$ such that
\[
\frac{1}{2}\leq\zeta(\theta)\leq\frac{p}{p+1}\quad\text{with}\quad\zeta(\theta)=\frac{(1-\theta)^{2}}{\theta}\frac{A_{k}\cdot L_{p}}{(p-1)!}\|z_{\theta}\|^{p-1}.
\]
Note that $\zeta(0)=+\infty$ and $\zeta(1)=0$. Hence, we can use
binary search to find $\theta$ that is close to $\theta^{*}$
such that $\zeta(\theta^{*})=\frac{7}{12}$ (or any value in $(\frac{1}{2},\frac{p}{p+1})$).
The main difficulty is to show how close $\theta$ need to be so that
$\zeta(\theta)\in[\frac{1}{2},\frac{p}{p+1}]$, or in other words to control the Lipschitz constant of $\zeta(\theta)$.

To bound the Lipschitz constant of $\zeta(\theta)$, we need to bound $\|z_{\theta}\|$
and $\|\frac{d}{d\theta}z_{\theta}\|$. First, we give an upper bound
on $\|\frac{d}{d\theta}z_{\theta}\|$.
\begin{lemma}
We have:
\label{lem:z_speed}
\[
\left\|\frac{d}{d\theta}z_{\theta}\right\|\leq5(p+1)^{2}\cdot\|x^{*}\|.
\]
\end{lemma}

\begin{proof}
To compute the derivative of $z_{\theta}$, we note by optimality condition that 
\[
\nabla_{z}F(z_{\theta},\widetilde{x}_{\theta})=0.
\]
Taking derivatives with respect to $\theta$ on both sides gives 
\[
\nabla_{zz}^{2}F(z_{\theta},\widetilde{x}_{\theta})\cdot\frac{d}{d\theta}z_{\theta}+\nabla_{zx}^{2}F(z_{\theta},\widetilde{x}_{\theta})\cdot\frac{d}{d\theta}\widetilde{x}_{\theta}=0.
\]
Hence, we have
\begin{equation}
\frac{d}{d\theta}z_{\theta}=-\left(\nabla_{zz}^{2}F(z_{\theta},\widetilde{x}_{\theta})\right)^{-1}\nabla_{zx}^{2}F(z_{\theta},\widetilde{x}_{\theta})\cdot(y_{k}-x_{k}).\label{eq:dz_dtheta}
\end{equation}
To bound $\frac{d}{d\theta}z_{\theta}$, it suffices to compute $\nabla_{zz}^{2}F(z,x)$
and $\nabla_{zx}^{2}F(z,x)$.

For $\nabla_{zz}^{2}F(z,x)$, we have
\[
\nabla_{zz}^{2}F(z,x) 
=\nabla_{zz}^{2}f_{p}(x+z,x)+\nabla^{2}\left[\frac{L_{p}}{p!}\|z\|^{p+1}\right].
\]
By doing a Taylor expansion of the Hessian function, one obtains:
\[
\|\nabla_{zz}^{2}f_{p}(x+z,x)-\nabla^{2}f(x+z)\|\leq\frac{L_{p}}{(p-1)!}\|z\|^{p-1}
\]
and hence
\[
\nabla_{zz}^{2}F(z,x)\succeq\nabla^{2}f(x+z)-\frac{L_{p}}{(p-1)!}\|z\|^{p-1}I+\frac{L_{p}(p+1)}{p!}\|z\|^{p-1}I\succeq\frac{L_{p}}{p!}\|z\|^{p-1}I
\]
where we used that $f$ is convex and 
\begin{equation}
\nabla^{2}\left[\|z\|^{p+1}\right]=(p+1)\|z\|^{p-1}\cdot I+(p+1)(p-1)\|z\|^{p-3}\cdot zz^{\top}.\label{eq:hessian_p_norm}
\end{equation}

For $\nabla_{zx}^{2}F(z,x)$, we recall that $F(z,x)=\sum_{i=0}^{p}\frac{1}{i!}D^{i}f(x)[z]^{i}+\frac{L_{p}}{p!}\|z\|^{p+1}$, and hence
\begin{align*}
\nabla_{zx}^{2}F(z,x) & =\sum_{i=1}^{p}\frac{1}{(i-1)!}D^{i+1}f(x)[z]^{i-1}\\
 & =\nabla_{zz}^{2}F(z,x)+\frac{1}{(p-1)!}D^{p+1}f(x)[z]^{p-1}-\nabla^{2}\left[\frac{L_{p}}{p!}\|z\|^{p+1}\right].
\end{align*}
Therefore, we have
\[
\left(\nabla_{zz}^{2}F(z,x)\right)^{-1}\left(\nabla_{zx}^{2}F(z,x)\right)=I+\left(\nabla_{zz}^{2}F(z,x)\right)^{-1}\left(\frac{D^{p+1}f(x)[z]^{p-1}}{(p-1)!}-\nabla^{2}\left[\frac{L_{p}}{p!}\|z\|^{p+1}\right]\right).
\]
and
\begin{align*}
\left\Vert \left(\nabla_{zz}^{2}F(z,x)\right)^{-1}\left(\nabla_{zx}^{2}F(z,x)\right)\right\Vert  & \leq1+\frac{p!}{L_{p}\|z\|^{p-1}}\left\Vert \frac{D^{p+1}f(x)[z]^{p-1}}{(p-1)!}-\nabla^{2}\left[\frac{L_{p}}{p!}\|z\|^{p+1}\right]\right\Vert \\
 & \leq1+\frac{p!}{L_{p}\|z\|^{p-1}}\left(\frac{L_{p}}{(p-1)!}\|z\|^{p-1}+\frac{L_{p}\cdot(p+1)p}{p!}\cdot\|z\|^{p-1}\right)\\
 & =(p+1)^{2}
\end{align*}
where we used (\ref{eq:hessian_p_norm}) and smoothness for the second inequality. Now, (\ref{eq:dz_dtheta})
and Lemma \ref{lem:diameter} below show
\[
\|\frac{d}{d\theta}z_{\theta}\|\leq(p+1)^{2}\cdot\|y_{k}-x_{k}\|\leq5(p+1)^{2}\cdot\|x^{*}\|.
\]
\end{proof}
\begin{lemma}
\label{lem:z_upper}We have that $\|z_{\theta}\|\leq12p^{3}\|x^{*}\|$
for all $0\leq\theta\leq1$.
\end{lemma}

\begin{proof}
By doing a Taylor expansion of the function $f$, one obtains:
\[
f_{p}(\widetilde{x}_{\theta}+z_{\theta},\widetilde{x}_{\theta})\geq f(\widetilde{x}_{\theta}+z_{\theta})-\frac{L_{p}}{(p+1)!}\|z_{\theta}\|^{p+1}.
\]
Hence, we have that
\begin{equation}
F(z_{\theta},\widetilde{x}_{\theta}) =f_{p}(\widetilde{x}_{\theta}+z_{\theta},\widetilde{x}_{\theta})+\frac{L_{p}}{p!}\|z_{\theta}\|^{p+1}
\geq f(\widetilde{x}_{\theta}+z_{\theta})+\frac{L_{p}\cdot p}{(p+1)!}\|z_{\theta}\|^{p+1}.\label{eq:F_lower}
\end{equation}
Rearranging the term, we have that 
\[
\|z_{\theta}\|^{p+1}\leq\frac{(p+1)!}{L_{p}\cdot p}\cdot(F(z_{\theta},\widetilde{x}_{\theta})-\min_{x}f(x))\leq\frac{(p+1)!}{L_{p}\cdot p}\cdot(f(\widetilde{x}_{\theta})-\min_{x}f(x))
\]
where we used that $F(z_{\theta},\widetilde{x}_{\theta})\leq F(0,\widetilde{x}_{\theta})=f(\widetilde{x}_{\theta})$.

For $\theta=1$, we have $\tilde{x}_{\theta}=y_{k}$ and hence
\[
\|z_{1}\|^{p+1}\leq\frac{(p+1)!}{L_{p}\cdot p}(f(y_{k})-\min_{x}f(x))\leq\frac{(p+1)!}{2p\cdot A_{k}\cdot L_{p}}\|x^{*}\|^{2}
\]
where we used \eqref{eq:tosatisfy2} at the end. Using Lemma \ref{lem:z_speed} and Young's inequality, we
have
\begin{align*}
\|z_{\theta}\| & \leq\left(\frac{(p+1)!}{2p\cdot A_{k}\cdot L_{p}}\right)^{\frac{1}{p+1}}\|x^{*}\|^{\frac{2}{p+1}}+5(p+1)^{2}\cdot\|x^{*}\|\\
 & \leq\frac{2}{p+1}\|x^{*}\|+\frac{p-1}{p+1}\left(\frac{(p+1)!}{2p\cdot A_{k}\cdot L_{p}}\right)^{\frac{1}{p-1}}+5(p+1)^{2}\|x^{*}\|.
\end{align*}
Using $A_{k}\geq\frac{k^{\frac{3p+1}{2}}}{c_{p}\cdot L_{p}\cdot\|x^{*}\|^{p-1}}\geq\frac{1}{c_{p}\cdot L_{p}\cdot\|x^{*}\|^{p-1}}$
and $c_{p}=\frac{2^{p-1}(p+1)^{\frac{3p+1}{2}}}{(p-1)!}$, we have
\[
\|z_{\theta}\|\leq\left(\frac{2}{p+1}+\frac{p-1}{p+1}\left(\frac{(p+1)!\cdot c_{p}}{2p}\right)^{\frac{1}{p-1}}+5(p+1)^{2}\right)\|x^{*}\|\leq12p^{3}\|x^{*}\|.
\]
\end{proof}
Next, we have a lower bound of $\|z_{\theta}\|$. We also prove Lipschitzness of $\theta \mapsto f(y_{\theta})$.
\begin{lemma}
\label{lem:z_lower}
We have
\[
\|z_{\theta}\|^p \geq \frac{p!}{L_p \cdot (p+2) \cdot (12 p^3 +4)\|x^*\|} (f(y_{\theta}) - f(x^*)) \,.
\]
Furthermore $\theta \mapsto f(y_{\theta})$ is Lipschitz, with Lipschitz constant upper bounded by
\[ 
L_{p}\cdot(12p^{3}\|x^{*}\|)^{p+1} \,.
\]
\end{lemma}

\begin{proof}
By the optimality of $z_{\theta}$, we have that
\[
\nabla_{z}f_{p}(\widetilde{x}_{\theta}+z_{\theta},\widetilde{x}_{\theta})+\frac{L_{p}\cdot(p+1)}{p!}\|z_{\theta}\|^{p-1}z_{\theta}=0.
\]
By doing a Taylor expansion of the gradient function, one obtains:
\[
\|\nabla_{z}f_{p}(\widetilde{x}_{\theta}+z_{\theta},\widetilde{x}_{\theta})-\nabla f(\widetilde{x}_{\theta}+z_{\theta})\|\leq\frac{L_{p}}{p!}\|z_{\theta}\|^{p}.
\]
Hence, we have $\|\nabla f(\widetilde{x}_{\theta}+z_{\theta})\|\leq\frac{L_{p}\cdot(p+2)}{p!}\|z_{\theta}\|^{p}$
and
\[
f(y_\theta) = f(\widetilde{x}_{\theta}+z_{\theta})\leq f(x^*)+\frac{L_{p}\cdot(p+2)}{p!}\|z_{\theta}\|^{p}\|\widetilde{x}_{\theta}+z_{\theta}-x^{*}\|.
\]
Since $\widetilde{x}_{\theta}$ is convex combination of $x_{k}$
and $y_{k}$, Lemma \ref{lem:diameter} shows that $\|\widetilde{x}_{\theta}-x^{*}\|\leq4\|x^{*}\|$
and Lemma \ref{lem:z_upper} shows that $\|z_{\theta}\|\leq12p^{3}\|x^{*}\|$.
Combining both, we have $\|\widetilde{x}_{\theta}+z_{\theta}-x^{*}\|\leq(12p^{3}+4)\|x^{*}\|$ and hence
$$f(y_{\theta})-f(x^{*})\leq\frac{L_{p}\cdot(p+2)}{p!}\|z_{\theta}\|^{p}\cdot(12p^{3}+4)\|x^{*}\|.$$
Rearranging gives the first inequality.

For the Lipschitz statement we note that, as above, we have:
\begin{align*}
f(y_{\theta})-f(y_{\theta'}) & \leq\frac{L_{p}\cdot(p+2)}{p!}\|z_{\theta}\|^{p}\|y_{\theta}-y_{\theta'}\|\\
 & \leq\frac{L_{p}\cdot(p+2)}{p!}\cdot(12p^{3}\|x^{*}\|)^{p}\cdot(\|\tilde{x}_{\theta}-\tilde{x}_{\theta'}\|+\|z_{\theta}-z_{\theta'}\|). 
\end{align*}
Lemma \ref{lem:diameter} shows that $\|\tilde{x}_{\theta}-\tilde{x}_{\theta'}\|=|\theta-\theta'|\cdot\|y_{k}-x_{k}\|\leq5\cdot\|x^{*}\|\cdot|\theta-\theta'|$.
Lemma \ref{lem:z_speed} shows that $\|z_{\theta}-z_{\theta'}\|\leq 5 (p+1)^{2}\|x^{*}\|\cdot|\theta-\theta'|$.
Combining both, we have
\[
f(y_{\theta})-f(y_{\theta'})\leq\frac{L_{p}\cdot(p+2)}{p!}\cdot(12p^{3}\|x^{*}\|)^{p}\cdot(5+5 (p+1)^{2})\|x^{*}\|\cdot|\theta-\theta'|.
\]
\end{proof}
We now give a bound on the Lipschitz constant $\zeta(\theta)$.
\begin{lemma}
\label{lem:lipschitz}
Denote 
\[
\omega_p (\theta) = 4(12p^{3})^{p+1}\cdot\left(1+A_{k}L_{p}\|x^{*}\|^{p-1}+\frac{L_{p}\|x^{*}\|^{p+1}}{\Delta(\theta)}\right) \,,
\]
and $\Delta(\theta) = f(y_{\theta}) - f(x^*)$. Then one has
\[
\left|\frac{d}{d\theta}\log\zeta(\theta)\right|\leq \omega_p (\theta) \cdot \left(1 + \frac{1}{\zeta(\theta)} + \zeta(\theta) \right)\,.
\]
\end{lemma}

\begin{proof}
Note that
\[
\frac{d}{d\theta}\log\zeta(\theta)=-\frac{2}{1-\theta}-\frac{1}{\theta}+(p-1)\frac{z_\theta \cdot \frac{d}{d\theta}z_{\theta}}{\|z_{\theta}\|^{2}}.
\]
Lemma \ref{lem:z_speed} shows that
\[
\left|\frac{d}{d\theta}\log\zeta(\theta)\right|\leq\frac{2}{1-\theta}+\frac{1}{\theta}+5(p+1)^{2}(p-1)\frac{\|x^{*}\|}{\|z_{\theta}\|}.
\]
The facts that
\[
\frac{1}{1-\theta}\leq1+\frac{\theta}{(1-\theta)^{2}} = 1+\frac{A_{k}\cdot L_{p}}{(p-1)! \cdot \zeta(\theta)} \|z_{\theta}\|^{p-1}
\]
and that
\[
\frac{1}{\theta}\leq2+\frac{(1-\theta)^{2}}{\theta}= 2+\frac{(p-1)! \cdot \zeta(\theta)}{A_{k}\cdot L_{p} \cdot\|z_{\theta}\|^{p-1}},
\]
yield:
\[
\left|\frac{d}{d\theta}\log\zeta(\theta)\right|\leq4+\frac{2 A_{k}\cdot L_{p}}{(p-1)! \cdot \zeta(\theta)}\|z_{\theta}\|^{p-1}+\frac{(p-1)! \cdot \zeta(\theta)}{A_{k}\cdot L_{p} \cdot \|z_{\theta}\|^{p-1}}+5(p+1)^{2}(p-1)\frac{\|x^{*}\|}{\|z_{\theta}\|} \,.
\]
It only remains to plug in Lemma \ref{lem:z_upper} and Lemma \ref{lem:z_lower} as follows:
For the second term, we have
\[
\frac{2A_{k}\cdot L_{p}}{(p-1)!\cdot\zeta(\theta)}\|z_{\theta}\|^{p-1}\leq\frac{2A_{k}\cdot L_{p}\cdot(12p^{3}\|x^{*}\|)^{p-1}}{\zeta(\theta)}\, .
\]
For the third term, we have
\begin{align*}
\frac{(p-1)!\cdot\zeta(\theta)}{A_{k}\cdot L_{p}\cdot\|z_{\theta}\|^{p-1}} & \leq\frac{(p-1)!\cdot12p^{3}\|x^{*}\|}{A_{k}\cdot L_{p}\cdot\|z_{\theta}\|^{p}}\cdot\zeta(\theta)\\
 & \leq\frac{(p-1)!\cdot12p^{3}\|x^{*}\|}{A_{k}\cdot L_{p}}\frac{L_{p}\cdot(p+2)\cdot(12p^{3}+4)\|x^{*}\|}{p!\cdot\Delta(\theta)}\cdot\zeta(\theta)\\
 & \leq4\cdot \frac{(12p^{3}\|x^{*}\|)^{2}}{A_{k}\cdot\Delta(\theta)}\cdot\zeta(\theta).
\end{align*}
Using $A_{k}\geq\frac{k^{\frac{3p+1}{2}}}{c_{p}\cdot L_{p}\cdot\|x^{*}\|^{p-1}}\geq\frac{1}{c_{p}\cdot L_{p}\cdot\|x^{*}\|^{p-1}}$
and $c_{p}=\frac{2^{p-1}(p+1)^{\frac{3p+1}{2}}}{(p-1)!}$, we have
\begin{align*}
\frac{(p-1)!\cdot\zeta(\theta)}{A_{k}\cdot L_{p}\cdot\|z_{\theta}\|^{p-1}} & \leq\frac{2^{p+1}(p+1)^{\frac{3p+1}{2}}}{(p-1)!}\frac{L_{p}\cdot\|x^{*}\|^{p-1}\cdot(12p^{3}\|x^{*}\|)^{2}}{\Delta(\theta)}\cdot\zeta(\theta)\\
 & \leq2^{p+1}(p+1)^{\frac{3p+1}{2}}(12p^{3})^{2}\cdot\frac{L_{p}\cdot\|x^{*}\|^{p+1}}{\Delta(\theta)}\cdot\zeta(\theta)\\
 & \leq4\cdot(12p^{3})^{p+1}\cdot\frac{L_{p}\cdot\|x^{*}\|^{p+1}}{\Delta(\theta)}\cdot\zeta(\theta).
\end{align*}
For the last term, we have
\begin{align*}
5(p+1)^{2}(p-1)\frac{\|x^{*}\|}{\|z_{\theta}\|} & \leq5(p+1)^{2}(p-1)\frac{(12p^{3}\|x^{*}\|)^{p-1}\|x^{*}\|}{\|z_{\theta}\|^{p}}\\
 & \leq5(p+1)^{3}\cdot(12p^{3}\|x^{*}\|)^{p-1}\cdot\frac{L_{p}\cdot(p+2)\cdot(12p^{3}+4)\|x^{*}\|^{2}}{p!\cdot\Delta(\theta)}\\
 & \leq4\cdot(12p^{3})^{p+1}\cdot\frac{L_{p}\cdot\|x^{*}\|^{p+1}}{\Delta(\theta)}.
\end{align*}
Combining all terms, we have the result
\[
\left|\frac{d}{d\theta}\log\zeta(\theta)\right|\leq4+\frac{2A_{k}\cdot L_{p}\cdot(12p^{3}\|x^{*}\|)^{p-1}}{\zeta(\theta)}+4\cdot(12p^{3})^{p+1}\cdot\frac{L_{p}\cdot\|x^{*}\|^{p+1}}{\Delta(\theta)}\cdot(\zeta(\theta)+1)\,
\]
justifying the claimed upper bound.
\end{proof}

The next lemma is a straightforward calculus exercise which allows to us to analyze binary search with guarantees of the form given in Lemma \ref{lem:lipschitz}.
\begin{lemma} \label{lem:calculusexercise}
Let $g : [0,1] \rightarrow \R_+$ and $\theta^* \in [0,1]$ such that $g(\theta^*) = \frac{7}{12}$. Let $\omega\geq0$ such that any $\theta \in [0,1]$ with $|\theta - \theta^*| \leq \frac{1}{40 \omega}$ satisfies
\[
\left| \frac{d}{d\theta} \log g(\theta) \right| \leq \omega \cdot \left(1 + \frac{1}{g(\theta)} + g(\theta) \right) \,.
\]
Then one also has $g(\theta) \in [\frac12, \frac23]$.
\end{lemma}

\begin{proof}
Let $h$ be the largest number such that $|\theta-\theta^{*}|\leq h$
implies $g(\theta)\in[\frac{1}{2},\frac{2}{3}]$. It suffices to show $h\geq\frac{1}{40\omega}$. Proceed by contradiction and suppose that $h\leq\frac{1}{40\omega}$. For
any $\theta$ such that $|\theta-\theta^{*}|\leq h$, by the assumption
on $g$ and $h$, we have
\[
\left|\frac{d}{d\theta}g(\theta)\right|\leq\omega\cdot(g(\theta)+1+g^{2}(\theta))\leq\omega\cdot\left(\frac{2}{3}+1+\left(\frac{2}{3}\right)^{2}\right)=\frac{19}{9}\omega.
\]
Hence, for any $\theta$ such that $|\theta-\theta^{*}|\leq h$, we
have $|g(\theta)-g(\theta^{*})|\leq h\cdot\frac{19}{9}\omega<\frac{1}{12}$.
Since $g$ is continuous and  $g(\theta^*) = \frac{7}{12}$ this contradicts the assumption of $h$ being the largest. Therefore $|\theta-\theta^{*}|\leq \frac{1}{40\omega}$ implies that $g(\theta) \in [\frac12, \frac23]$ as desired.
\end{proof}
Now, we can prove our main theorem of this section. 

\begin{theorem}
Let $\epsilon > 0$. At iteration $k$, using at most $30p\log_{2}p+\log_{2}\left\lceil \frac{L_{p}\|x^{*}\|^{p+1}}{\epsilon}\right\rceil$ calls to the $p^{th}$ order Taylor oracle we find either a point $y$ such that $f(y) - f(x^*) \leq \epsilon$ or we find $\lambda_{k+1}$ that satisfies \eqref{eq:key4}.
\end{theorem}

\begin{proof}
First note that we can assume $A_k \leq \|x^*\|^2 / (2 \epsilon)$, for otherwise $f(y_k) - f(x^*) \leq \epsilon$ by Lemma \ref{lem:basic2}. Now using $\log_2(1/\delta)$ binary search step on $\zeta$, let us find $\theta$ such that $|\theta - \theta^*| \leq \delta$ for some $\theta^*$ with $\zeta(\theta^*) = \frac{7}{12}$. \\
If $\Delta(\theta) \leq \epsilon$ then we are done, so let us assume this is not the case. By the Lipschitz constant bound from Lemma \ref{lem:z_lower}, as well as choosing $\delta$ smaller than $\epsilon/2$ divided by this Lipschitz constant, we obtain that $\Delta(\theta') \geq \epsilon /2$ for any $\theta'$ such that $|\theta - \theta'| \leq 2 \delta$ (so in particular for any $\theta'$ such that $|\theta' - \theta^*|\leq \delta$). We now want to apply Lemma \ref{lem:calculusexercise} to conclude that $\zeta(\theta) \in [\frac{1}{2}, \frac{2}{3}]$. For this we need to compute a value for $\omega$ using Lemma \ref{lem:lipschitz} (and we will want $\delta$ small enough so that $\delta \leq \frac{1}{40 \omega}$). One can easily verify that the following value of $\omega$ works given the above:
\begin{align*}
\omega & \leq4(12p^{3})^{p+1}\cdot\left(1+A_{k}L_{p}\|x^{*}\|^{p-1}+\frac{L_{p}\|x^{*}\|^{p+1}}{\Delta(\theta)}\right)\\
 & \leq4(12p^{3})^{p+1}\cdot\left(1+\frac{\|x^{*}\|^{2}}{2\epsilon}L_{p}\|x^{*}\|^{p-1}+\frac{L_{p}\|x^{*}\|^{p+1}}{\epsilon/2}\right)\\
 & \leq16\cdot(12p^{3})^{p+1}\cdot\left\lceil \frac{L_{p}\|x^{*}\|^{p+1}}{\epsilon}\right\rceil .
\end{align*}
Hence we can choose
$$\frac{1}{\delta}=640\cdot(12p)^{3(p+1)}\cdot\left\lceil \frac{L_{p}\|x^{*}\|^{p+1}}{\epsilon}\right\rceil \leq p^{30p}\cdot\left\lceil \frac{L_{p}\|x^{*}\|^{p+1}}{\epsilon}\right\rceil\, $$
and binary search finishes in $\log_2(1/\delta) = 30p\log_{2}p+\log_{2}\left\lceil \frac{L_{p}\|x^{*}\|^{p+1}}{\epsilon}\right\rceil$ steps.
\end{proof}
Finally, we give the bound for $\|x_{k}-x^{*}\|$ and $\|y_{k}-x^{*}\|$.
\begin{lemma}
\label{lem:diameter} We have that $\|x_{k}-x^{*}\|\leq\|x^{*}\|$
and $\|y_{k}-x^{*}\|\leq4\|x^{*}\|$ for all $k$.
\end{lemma}

\begin{proof}
From Lemma \ref{lem:basic4} we have
\[\psi_{k+1}(x_{k+1}) - A_{k+1} f(y_{k+1}) \geq \sum_{i=1}^{k+1} \frac{A_{i}}{2 \lambda_{i}} \bigg( (1-\sigma^2) \|y_{i} - \tilde{x}_{i-1}\|^2\bigg) \]
Since from Lemma \ref{lem:basic1}
\[\psi_{k+1}(x_{k+1}) +\frac{1}{2}\|x^*-x_{k+1}\|^2 = \psi_{k+1}(x^*) \leq A_{k+1} f(x^*)+\frac{1}{2}\|x^*\|^2\]
altogether this gives
\begin{align*}
\sum_{i=1}^{k+1} \frac{A_{i}}{2 \lambda_{i}} \bigg( (1-\sigma^2) \|y_{i} - \tilde{x}_{i-1}\|^2\bigg) &\leq A_{k+1}(f^*-f(y_{k+1}))+\frac{1}{2}\|x^*\|^2-\frac{1}{2}\|x^*-x_{k+1}\|^2
\end{align*}
therefore we have that $\|x_{k}-x^{*}\|\leq\|x^{*}\|$ for all $k$. Let $D_{k}=\|y_{k}-x^{*}\|$. Using $\widetilde{x}_{k}=\frac{A_{k}}{A_{k+1}}y_{k}+\frac{a_{k+1}}{A_{k+1}}x_{k}$,
we have
\[
\|\widetilde{x}_{k}-x^{*}\|\leq\frac{A_{k}}{A_{k+1}}D_{k}+\frac{a_{k+1}}{A_{k+1}}\|x^{*}\|.
\]
Hence, we have $D_{k+1}\leq\frac{A_{k}}{A_{k+1}}D_{k}+\frac{a_{k+1}}{A_{k+1}}\|x^{*}\|+\|y_{k+1}-\widetilde{x}_{k}\|$. Rescaling and summing over $k$, we have 
\begin{align*}
D_{k+1} & \leq\|x^{*}\|+\|y_{k+1}-\widetilde{x}_{k}\|+\frac{A_{k}}{A_{k+1}}\|y_{k}-\widetilde{x}_{k-1}\|+\frac{A_{k-1}}{A_{k+1}}\|y_{k-1}-\widetilde{x}_{k-2}\|+\cdots\\
 & \leq\|x^{*}\|+\frac{1}{A_{k+1}}\sum_{j=1}^{k+1}A_{j}\|y_{j}-\widetilde{x}_{j-1}\|\\
 & \leq\|x^{*}\|+\frac{\sqrt{\sum_{j=1}^{k+1}A_{j}\lambda_{j}}}{A_{k+1}}\sqrt{\sum_{j=1}^{k+1}\frac{A_{j}}{\lambda_{j}}\|y_{j}-\widetilde{x}_{j-1}\|^{2}}\\
 & \leq\|x^{*}\|+\frac{\sqrt{\sum_{j=1}^{k+1}\lambda_{j}}}{\sqrt{A_{k+1}}}\sqrt{\frac{\|x^{*}\|^{2}}{1-\sigma^{2}}}\\
 & \leq4\|x^{*}\|
\end{align*}
where we used $A_{j}$ is increasing and \eqref{eq:Alambdatradeoff} in the second to last
equation, and Lemma \ref{lem:Aestimate} and $\sigma=\frac{1}{2}$ for the last.
\end{proof}

\bibliographystyle{plainnat}
\bibliography{bib}

\end{document}